\documentclass{amsart}
\usepackage{graphicx}
\usepackage{amssymb}
\usepackage{array}
\usepackage{enumerate}
\usepackage{url}
\usepackage{algorithm}

\newtheorem{theorem}{Theorem}

\newtheorem{lemma}[theorem]{Lemma}

\numberwithin{theorem}{section}
\theoremstyle{remark}

\begin{document}

\title[A Deterministic factoring algorithm]{A deterministic algorithm for
integer factorization}
\author[G.A. Hiary]{Ghaith A. Hiary}
\thanks{Preparation of this material is partially supported by
the National Science Foundation under agreements No. 
 DMS-1406190 and by 
the Leverhulme Trust (while at the University of Bristol).}
\address{Department of Mathematics, The Ohio State University, 231 West 18th
Ave, Columbus, OH 43210.}
\email{hiaryg@gmail.com}

\begin{abstract}
A deterministic algorithm for factoring $n$ using $n^{1/3+o(1)}$ 
bit operations is presented. The algorithm tests 
the divisibility of $n$ by all the integers in 
a short interval at once, rather than integer by integer as in 
trial division. The algorithm is implemented.
\end{abstract}

\maketitle

\section{Introduction} \label{intro}

One can use trial division to factor an integer $n$ using 
$\le n^{1/2}$ divisions on integers of size $\le n$.
There are several algorithms that improve the running time to
$n^{1/3+o(1)}$ bit operations without using fast Fourier
transform (FFT) techniques: Lehman's method~\cite{Lehman} which ``uses
a dissection of the continuum similar to the Farey dissection,''
Lenstra's algorithm~\cite{Lenstra} which looks for divisors of $n$ in residue
classes, McKee's algorithm~\cite{McKee} which 
is related to Euler's factoring method, 
and an algorithm due to Rubinstein~\cite{Rubinstein} that relies on estimates
for Kloosterman sums. 
The Pollard-Strassen algorithm~\cite{Pollard,Strassen} 
uses an FFT precomputation  
to improve the time complexity to $n^{1/4+o(1)}$ bit operations and
requiring $n^{1/4+o(1)}$ bits of storage (memory); see \cite{BGS,CH,CP,Vas} 
for example. As far as we know, this is the fastest 
deterministic factoring algorithm with a fully proven complexity, though 
 it has the practical disadvantage of requiring much memory space.
The Coppersmith algorithm~\cite{coppersmith} for finding small roots of bivariate
rational polynomials enables factoring $n$ using $n^{1/4+o(1)}$ operations. 
This algorithm uses lattice basis reduction techniques,
and it has the advantage of requiring little memory space.
(However, the $n^{o(1)}$ factor in the running time seems significant, 
involving a high power of $\log n$.)
Shank's class group method (see~\cite{CP}) has 
a better complexity of $n^{1/5+o(1)}$ bit operations to factor $n$,
but it assumes the generalized Riemann hypothesis, which is unproved so far.

In this paper, we present a new deterministic method for
factoring $n$ in $n^{1/3+o(1)}$ time. 
Like the other exponential factoring methods mentioned before, this algorithm 
is mainly of a theoretical interest. There 
are already probabilistic methods that far outperform 
it in practice, and in heuristically subexponential time; see \cite{CP}
for a survey of such methods. Our goal, rather, is to 
present a new deterministic approach for integer
factorization that we hope can be improved in the future.

\section{Main result} 
An integer $n> 1$ is composite if the equation $n=x y$ has a non-trivial 
integer solution $(x,y)$. One can test whether $n=xy$ holds 
by testing if $n/x \equiv 0 \bmod{1}$, which can be decided on
dividing $n$ by $x$ directly, say. 
By looping through the integers $1<x\le \sqrt{n}$ this way, 
one will either find a non-trivial factor of $n$, or, if no 
such factor is found, 
one concludes that $n$ is prime. This trial division 
procedure is guaranteed to terminate after $\le \sqrt{n}$ steps. 
The new algorithm that we present, Algorithm~\ref{onethirdalg}, enables
a speed-up over trial division because it can test the equation 
$n/(x+h)\equiv 0\bmod{1}$ for many integers $h\in [-H, H]$ in basically a single step. The observation is that, locally (i.e.\ if $H$ is small enough 
compared to $x$), one can approximate 
 $n/(x+h)$ by a linear polynomial in $h$ with rational coefficients. 
The oscillations of this polynomial modulo $1$ are easy
to understand, due to linearity, which leads to the speed-up.

The main result is Theorem~\ref{onethird}, which gives an upper bound
on the complexity of Algorithm~\ref{onethirdalg}. 
The complexity is measured by
the total number of the following operations consumed:
$+,-, \times,\div,\exp,\log$. 
This in turn can be routinely bounded in terms of bit operations
since all the numbers that occur in Theorem~\ref{onethird}
can be expressed using $\ll \log n$ bits.
We will make use of some basic algorithms such as 
how to generate the continued fraction (CF) 
convergents of a rational number and 
how to solve a quadratic equation.
We will use the notation $[x]$
to denote the nearest integer to $x$ (if $x$ is half an integer, we take
$[x]=\lfloor x\rfloor$). 
\begin{algorithm}
\caption{Given an integer $n>1$, this algorithm finds 
a non-trivial factor of $n$ or proves that $n$ is prime.}\label{onethirdalg}
\begin{enumerate}
\tt
\item[1.][Initialize]\\
set $x_0=\min\{\lceil (17n)^{1/3}\rceil,\lfloor \sqrt{n}\rfloor\}$, 
$x=x_0+2$, $H=1$;\\
\vspace{1mm}
\item[2.][Trial division]\\
check if $n$ has a divisor $1< k\le x_0$, if so return $k$;\\
\vspace{1mm}
\item[3.][Loop]\\
while($x-H\le \lfloor \sqrt{n}\rfloor$) \{\\
\vspace{1mm}\noindent\leftskip=8mm
generate the CF convergents of $n/x^2$, say
$[b_0/q_0,\ldots,b_r/q_r]$,
then find the convergent with the largest $q_j\le 4H$;\\
\vspace{1mm}\noindent\leftskip=8mm
set $b=b_j$, $q=q_j$, $a=[qn/x]$;\\ 
\vspace{1mm}\noindent\leftskip=8mm
solve $(qn-ax)+(bx-a)h+bh^2=0$; 
for each integer solution $h$ 
test if $x+h$ divides $n$, if so return $x+h$;\\ 
\vspace{1mm}\noindent\leftskip=8mm
increment $x\leftarrow x+2H+1$, set $H=\lfloor (17n)^{-1/3}x\rfloor$;\\
\vspace{1mm}\noindent\leftskip=3mm
\}\\
\vspace{1mm}\noindent\leftskip=3mm
return $n$ is prime;
\end{enumerate}
\end{algorithm}
\begin{theorem}\label{onethird}
Algorithm~\ref{onethirdalg} returns a non-trivial factor of $n>1$, or
proves that $n$ is prime, using $\ll n^{1/3}\log^2 n$
operations on numbers of $\ll \log n$ bits.
\end{theorem}
The while loop in Algorithm~\ref{onethirdalg} checks for divisors of $n$ in
successive blocks of the form $[x-H,x+H]$. 
The block size, $2H+1$, increases as the loop progresses. Roughly speaking,
as the algorithm searches through an interval like $[x,2x]$, 
$H$ doubles in size, increasing from $H\approx x/(17n)^{1/3}$ at the beginning, 
to $H\approx 2x/(17n)^{1/3}$ by the end.
This choice of $H$ is not optimal, in that it can be chosen larger
depending on $a$, $b$, and $q$;
see  \textsection{\ref{implementation}}.
However, fixing the choice like we did 
 simplifies the proof of Theorem~\ref{onethird} later.

One feature of Algorithm~\ref{onethird} is that,
like the Pollard-Strassen method, it 
can be adapted to obtain partial information
about the factorization of $n$. 
For example, after small modifications, 
Algorithm~\ref{onethirdalg} can rule out factors of $n$ in 
a given interval $[z,z+w]$, $z,w\in\mathbb{Z}^+$,
using $\ll (wn^{1/3}/z+1)\log(n+z+w)$ operations. To do so, one
adjusts the trial division statement to cover
the smaller range $\tt z\le k\le \min\{x_0,z+w\}$, 
then initializes $\tt x=x_0+2$ or $\tt x=z+1$
depending on whether $z<x_0$ or not, and
adjusts the loop statement to be $\tt while(x-H\le z+w)$. 

It is interesting to compare our method with the Coppersmith algorithm.\footnote{It 
is puzzling that the Coppersmith method does not seem
to be more well-known in the relevant number theory literature.} 
Using the latter, one can factor
 $n=pq$ in poly-log time in $n$ 
if the high-order $\frac{1}{4} \log_2 N$ bits of $p$ are known.\footnote{In fact, 
the Coppersmith algorithm leads to 
the same result if the low-order bits of $p$ are known, instead of the
high-order bits. More generally, 
the algorithm can decide, in poly-log time in $n$, 
whether $p$ lies in a given
residue class modulo an integer of size about $n^{1/4}$.} (Here, 
$\log_2$ is the logarithm to base $2$.)
By comparison, our method requires more, 
the high-order $\frac{1}{3} \log_2 N$ bits of $p$. 
Therefore, our method is of a comparable strength 
to the algorithm of Rivest and Shamir~\cite{RivestShamir}, 
where this problem is set up in terms of integer programming 
in two dimensions.

\section{Proof of Theorem~\ref{onethird}}
\begin{lemma}\label{cf}
Let $n$, $x$, and $H$ be positive integers. 
Then there is a rational 
approximation of $n/x^2$ of the form
\begin{equation}\label{cfapprox}
\frac{n}{x^2} = \frac{b}{q} +\frac{\epsilon_2}{qq'},
\quad 0<q\le 4H\le q', \quad |\epsilon_2|< 1. 
\end{equation}
This approximation can be found using $\ll \log(n+x)$ operations
on integers of $\ll \log(n+x)$ bits.
\end{lemma}
\begin{proof}
This follows routinely from the classical theory of continued fractions; 
see \cite{davenport} for example. 
\end{proof}
\begin{lemma}\label{congruence}
Let $n\ge 400$, $x$, and $H$ be positive integers
with $H/x \le (17n)^{-1/3}$. For each integer $|h|\le H$,
if $n/(x+h) \equiv 0\bmod{1}$ then $h$ 
must be a solution of the equation 
$g_{n,x}(y):=c_0+c_1 y + c_2 y^2=0$ where, letting
\begin{equation}\label{eps1}
\frac{n}{x}=\frac{a}{q}+\frac{\epsilon_1}{q},\quad a=[qn/x],
\end{equation}
we have $c_0 := x\epsilon_1=qn-ax$, $c_1 := \epsilon_1 - x\epsilon_2 /q'=bx-a$,
and $c_2:=nq/x^2 -\epsilon_2/q'=b$. 
(Here, $q$, $q'$, and $\epsilon_2$ are as in Lemma~\ref{cf}.)
Moreover, $g_{n,x}(y)$ does not vanish identically, so there are at most two
solutions of the equation $g_{n,x}(y)=0$.
\end{lemma}
\begin{proof}
Since $x+h>0$ for $|h|\le H$, we have the identity:
$n/(x+h) = n/x-n h/x^2 +n h^2/((x+h)x^2)$.
Let us define $\epsilon(h) := \epsilon_1 -h \epsilon_2/q' + q n
h^2/((x+h)x^2)$. Then
\begin{equation}
\frac{n}{x+h} = \frac{a-bh}{q} + \frac{\epsilon(h)}{q}.
\end{equation}
Multiplying both sides by $q$, we see that 
if $n/(x+h)\equiv 0\bmod{1}$ then necessarily 
$a-bh+\epsilon(h)\equiv 0 \bmod{q}$. In particular, since
$a-bh$ is an integer, so must $\epsilon(h)$; i.e.\
$\epsilon(h) \equiv 0 \bmod{1}$. By the triangle inequality, 
the bound $|h|\le H$, and the bound $H\le x/2$, we have
\begin{equation}
|\epsilon(h)|\le |\epsilon_1|+ \left|\frac{\epsilon_2 H}{q'}\right|+ 
\frac{qnH^2}{x^2(x-H)}.
\end{equation}
By construction, $|\epsilon_1|\le 1/2$, $|\epsilon_2 H/q'|< 1/4$, and $q\le 4H$. 
Since also $H\le (17n)^{-1/3}x$ by hypothesis,
we obtain that $qnH^2/(x^2(x-H))\le  4n(H/x)^3/(1-H/x) \le (4/17)/(1-6800^{-1/3})<
1/4$, where we used the assumption $n\ge 400$. 
So we deduce that $|\epsilon(h)|<1/2+1/4+1/4=1$. Therefore,
in our situation, the congruence 
$\epsilon(h)\equiv 0\bmod{1}$ is equivalent to the equation
$\epsilon(h) = 0$.
Last, since $(x+h)\epsilon(h)=g_{n,x}(h)$, and $x+h\ne 0$, 
we deduce that the equations $\epsilon(h)=0$
and $g_{n,x}(h)=0$ are equivalent.

For the second part of the lemma, note that if $g_{n,x}(y)\equiv 0$, then
$c_0=c_1=c_2=0$. Since $c_0=0$ then $\epsilon_1=0$. 
And since $c_1=0$ also, we deduce that
 $\epsilon_2=0$. But then $c_2=nq/x^2-\epsilon_2/q'=nq/x^2\ne 0$.
\end{proof}

\begin{proof}[Proof of Theorem~\ref{onethird}]
We choose integers $1<x_0<x_1\cdots$, and $1\le H_1\le H_2\le \ldots$, 
and define the following sequence of intervals: 
$B_0:=[2,x_0]$, $B_1:=[x_1-H_1,x_1+H_1]$, $B_2:=[x_2-H_2,x_2+H_2],\ldots$.
Specifically, we 
choose $x_0:=\lceil (17n)^{1/3}\rceil$, $x_1:=x_0+2$, $H_1=1$, 
and, for $j\ge 2$, we let $x_j:=x_{j-1}+2H_{j-1}+1$
where $H_j:= \lfloor (17n)^{-1/3}x_j\rfloor$. 
So $1\le H_1\le H_2\le \cdots$, and therefore $B_0\cup \cdots\cup B_j$
 covers the interval $[2,x_j]$ completely.
We use trial division to search for 
a factor of $n$ in $B_0$ using $\ll n^{1/3}$ operations.
If a factor is found, then it is returned and 
the algorithm reaches an end point. Otherwise,
we successively search for a factor in the intervals $B_j=[x_j+H_j,x_j-H_j]$. 
We note at this point that if $n< 400$ then the algorithm reaches an end
after searching $B_0$. This is because $x_0= \lceil(17n)^{1/3}\rceil \ge \lfloor
\sqrt{n}\rfloor$ for $n<400$, as can be checked by direct computation, 
and this implies that the algorithm will not enter the while loop.
So, in analyzing the Loop phase of the algorithm, we may assume that $n\ge 400$. 
Furthermore, we observe that $H_j/x_j \le (17n)^{-1/3}$ by construction, and so
 $n$, $x_j$, and $H_j$ satisfy the hypothesis
of Lemma~\ref{congruence}. Thus, applying the Lemma
to $B_j$, one can quickly locate all the divisors of $n$ in that block (if any)
using $\ll \log (n+x_j)$ operations on numbers
of $\ll \log (n+x_j)$ bits. This is mainly the cost of finding the rational
approximation in Lemma~\ref{cf} via 
the continued fraction representation of $n/x^2$, then solving
the resulting quadratic equation.
Last, we only need to search $B_j$ that satisfy $B_j \cap [2,\lfloor
\sqrt{n}\rfloor]\ne \varnothing$; i.e.\ $x_j-H_j\le  \sqrt{n}$.
This is because if 
no factor is found in these $B_j$, then one will have proved $n$ prime.
Given this, it is easy 
to show that the total number of blocks that 
need to be searched is $\ll n^{1/3} \log(2+n/x_0)$. 
Since $n\ge 2$ by hypothesis, this is $\ll n^{1/3}\log n$, 
which yields the result.
\end{proof}

\section{Implementation}\label{implementation}
We implemented Algorithm~\ref{onethirdalg} 
in \verb!Mathematica!. The implementation
is available at \url{https://people.math.osu.edu/hiary.1/factorTest.nb}.
We were able to reduce the running time by about $20\%$ by choosing the block size
asymmetrically about $x$. From the left 
we set $H_L=\lfloor (17n)^{-1/3} x\rfloor$, which is the same 
as in Algorithm~\ref{onethirdalg},
and from the right we set $H_R=\min\{H_{R,1},H_{R,2}\}$ where 
$H_{R,1}=\lfloor 0.4(1-|\epsilon_1|)q'/|\epsilon_2|\rfloor$ and
$H_{R,2}=\lfloor \sqrt{0.6(1-|\epsilon_1|)x^3/(qn))}\rfloor$. 
Also, we required that $q\le 4H_L$. 
Together, this ensured that $|\epsilon(h)|<1$ for $-H_L\le h\le H_R$, 
as needed, and it allowed a larger block size.
This is because $H_R$ will be at least the size of $H_L$, 
but it can get much larger if $|\epsilon_2/q'|$ and $q$ happened to be small; e.g.\
if $n/x^2$ can be approximated well by a rational with a small denominator. 
To take advantage of the larger block size in the implementation, 
we incremented $x\leftarrow x + H_L+H_R+1$ instead of $x\leftarrow x+2H_L+1$.

Our implementation
of Lemma~\ref{cf} became faster, on average, than trial division when $H_L\gtrsim 50$.
So we used trial division in the interval $[2,\lceil 50 (17n)^{1/3}]\rceil]$.
The running time of the full algorithm started to 
beat trial division when $n\gtrsim 10^{14}$,
with $n$ a product of two primes of roughly equal size.
The algorithm is about two times faster than 
trial division when $n\approx 10^{18}$.
This running time can be expected to improve 
using a more careful implementation; e.g.\ one need not generate all the
continued fraction convergents of $n/x^2$, as done now, 
but only the convergents with denominator $\le 4H_L$.

\bibliographystyle{amsplain}
\bibliography{factor}

\providecommand{\bysame}{\leavevmode\hbox to3em{\hrulefill}\thinspace}
\providecommand{\MR}{\relax\ifhmode\unskip\space\fi MR }
\providecommand{\MRhref}[2]{%
  \href{http://www.ams.org/mathscinet-getitem?mr=#1}{#2}
}
\providecommand{\href}[2]{#2}
\begin{thebibliography}{10}

\bibitem{BGS}
Alin Bostan, Pierrick Gaudry, and {\'E}ric Schost, \emph{Linear recurrences
  with polynomial coefficients and application to integer factorization and
  {C}artier-{M}anin operator}, SIAM J. Comput. \textbf{36} (2007), no.~6,
  1777--1806. \MR{2299425 (2008a:11156)}

\bibitem{coppersmith}
Don Coppersmith, \emph{Finding a small root of a bivariate integer equation;
  factoring with high bits known}, Advances in cryptology---{EUROCRYPT} '96
  ({S}aragossa, 1996), Lecture Notes in Comput. Sci., vol. 1070, Springer,
  Berlin, 1996, pp.~178--189. \MR{1421585 (97h:94009)}

\bibitem{CH}
Edgar Costa and David Harvey, \emph{Faster deterministic integer
  factorization}, Math. Comp. \textbf{83} (2014), no.~285, 339--345.
  \MR{3120593}

\bibitem{CP}
Richard Crandall and Carl Pomerance, \emph{Prime numbers}, second ed.,
  Springer, New York, 2005, A computational perspective. \MR{2156291
  (2006a:11005)}

\bibitem{davenport}
H.~Davenport, \emph{The higher arithmetic}, eighth ed., Cambridge University
  Press, Cambridge, 2008, An introduction to the theory of numbers, With
  editing and additional material by James H. Davenport. \MR{2462408
  (2009j:11001)}

\bibitem{Lehman}
R.~Sherman Lehman, \emph{Factoring large integers}, Math. Comp. \textbf{28}
  (1974), 637--646. \MR{0340163 (49 \#4919)}

\bibitem{Lenstra}
H.~W. Lenstra, Jr., \emph{Divisors in residue classes}, Math. Comp. \textbf{42}
  (1984), no.~165, 331--340. \MR{726007 (85b:11118)}

\bibitem{McKee}
James McKee, \emph{Turning {E}uler's factoring method into a factoring
  algorithm}, Bull. London Math. Soc. \textbf{28} (1996), no.~4, 351--355.
  \MR{1384821 (97f:11010)}

\bibitem{Pollard}
J.~M. Pollard, \emph{Theorems on factorization and primality testing}, Proc.
  Cambridge Philos. Soc. \textbf{76} (1974), 521--528. \MR{0354514 (50 \#6992)}

\bibitem{RivestShamir}
Ronald~L. Rivest and Adi Shamir, \emph{Efficient factoring based on partial
  information}, Advances in cryptology---{EUROCRYPT} '85 ({L}inz, 1985),
  Lecture Notes in Comput. Sci., vol. 219, Springer, Berlin, 1986, pp.~31--34.
  \MR{851581}

\bibitem{Rubinstein}
Michael~O. Rubinstein, \emph{The distribution of solutions to {$xy=n\pmod a$}
  with an application to factoring integers}, Integers \textbf{13} (2013),
  Paper No. A12, 20. \MR{3083474}

\bibitem{Strassen}
Volker Strassen, \emph{Einige {R}esultate \"uber {B}erechnungskomplexit\"at},
  Jber. Deutsch. Math.-Verein. \textbf{78} (1976/77), no.~1, 1--8. \MR{0438807
  (55 \#11713)}

\bibitem{Vas}
O.~N. Vasilenko, \emph{Number-theoretic algorithms in cryptography},
  Translations of Mathematical Monographs, vol. 232, American Mathematical
  Society, Providence, RI, 2007, Translated from the 2003 Russian original by
  Alex Martsinkovsky. \MR{2273200 (2007g:11160)}

\end{thebibliography}
\end{document}